\documentclass[
11pt,twoside, a4paper,
]{amsart}

\usepackage{amsmath,amsfonts,amssymb}
\usepackage[all]{xy}
\usepackage{mathptmx}
\usepackage{eucal}
\usepackage{amsxtra}
\usepackage{verbatim}
\usepackage{url}
\usepackage{enumerate}
\usepackage{mathtools}
\usepackage{dsfont}

\title[The first and second homotopy group of a homogeneous space]%
{The first and second homotopy groups\\ of a homogeneous space  \\
of a complex linear algebraic group}

\author{Mikhail Borovoi}

\address{
Raymond and Beverly Sackler School of Mathematical Sciences,
Tel Aviv University, 6997801 Tel Aviv, Israel}
\email{borovoi@tauex.tau.ac.il}

\thanks{The author was partially supported
by the Israel Science Foundation (grant 1030/22).}



\DeclareSymbolFont{rsfs}{U}{rsfs}{m}{n}
\DeclareSymbolFontAlphabet{\mathrsfs}{rsfs}

\DeclareFontEncoding{OT2}{}{} % to enable usage of cyrillic fonts
\DeclareTextFontCommand{\textcyr}{\fontencoding{OT2}
    \fontfamily{wncyr}\fontseries{m}\fontshape{n}\selectfont}

\theoremstyle{plain}

\newtheorem{theorem}{Theorem} [section]

\newtheorem{lemma}[theorem]{Lemma}
\newtheorem{corollary}[theorem]{Corollary}

\newtheorem{conditional-result}[theorem]{Conditional Result}

\newtheorem{theorem?}{Theorem(?)} [section]
\newtheorem{proposition?}[theorem]{Proposition(?)}
\newtheorem{lemma?}[theorem]{Lemma(?)}
\newtheorem{corollary?}[theorem]{Corollary(?)}

\newtheorem*{theorem*}{Theorem}
\newtheorem*{proposition*}{Proposition}
\newtheorem*{lemma*}{Lemma}
\newtheorem*{corollary*}{Corollary}
\newtheorem*{question*}{Question}
\newtheorem*{conjecture*}{Conjecture}
\newtheorem*{claim*}{Claim}

\newtheorem*{introtheorem*}{Theorem}
\newtheorem*{introproposition*}{Proposition}
\newtheorem*{introlemma*}{Lemma}
\newtheorem*{introcorollary*}{Corollary}

\theoremstyle{definition}

\newtheorem{remark}[theorem]{Remark}

\newtheorem*{definition*}{Definition}
\newtheorem*{example*}{Example}

\theoremstyle{remark}

\newcommand{\labelt}[1]{\xrightarrow{\makebox[1.2em]{\scriptsize ${#1}$}}}
\newcommand{\labelto}[1]{\xrightarrow{\makebox[1.5em]{\scriptsize ${#1}$}}}

\newcommand{\lra}{\longrightarrow}

\newcommand{\isoto}{\overset{\sim}{\to}}
\newcommand\longisoto{\overset\sim\lra}

\newcommand{\into}{\hookrightarrow}
\newcommand{\onto}{\twoheadrightarrow}

\def\uu{\mathrm{u}}
\def\red{\mathrm{red}}
\def\tor{{\mathrm{tor}}}
\def\ssc{{\mathrm{sc}}}
\def\sss{{\mathrm{ss}}}
\def\mult{{\mathrm{mult}}}
\def\ssu{{\mathrm{ssu}}}

\renewcommand{\mathbb}{\mathds}

\newcommand{\RR}{{\mathbb{R}}}
\newcommand{\CC}{{\mathbb{C}}}
\newcommand{\ZZ}{{\mathbb{Z}}}
\newcommand{\QQ}{{\mathbb{Q}}}

\newcommand{\Gal}{{\rm Gal}}
\newcommand{\Pic}{{\rm Pic}}

\renewcommand{\ker}{{\rm ker}}
\newcommand{\coker}{{\rm coker}}

\newcommand{\Hom}{{\rm Hom}}

\newcommand{\That}{{\widehat{T}}}
\newcommand{\Hhat}{{\widehat{H}}}

\newcommand{\Ghat}{{\widehat{G}}}

\def\G{{\mathbb{G}}}

\def\Hhat{{\widehat{H}}}
\def\Z{{\ZZ}}
\def\Q{{\QQ}}
\def\C{{\CC}}
\def\R{{\RR}}
\def\Ext{{\textup{Ext}}}

\def\top{^{\textup{top}}}
\def\alg{_{\textup{alg}}}

\def\top{{\textup{top}}}

\def\alg{{\textup{alg}}}
\def\kerchar{{\textup{\rm ker.char}}}
\def\Stab{{\textup{Stab}}}

\def\Zhat{{\widehat{\Z}}}

\def\Kbul{{K^\bullet}}
\def\tors{{\rm tors}}
\def\tf{{\rm t.f.}}
\def\AA{{\mathrsfs{A}}}

\def\ii{{\textbf{\textit{i}}}}
\def\pit{\pi^\top}

\renewcommand{\Zhat}{{\Z^\wedge}}
\newcommand{\piet}{{\pi_1^{\text{\rm\'et}}}}

\newcommand{\hs}{\kern 0.8pt}
\newcommand{\hssh}{\kern 1.2pt}
\newcommand{\hshs}{\kern 1.6pt}
\newcommand{\hssss}{\kern 2.0pt}

\newcommand{\hm}{\kern -0.8pt}
\newcommand{\hmm}{\kern -1.2pt}

\newcommand{\GmC}{{\G_{{\rm m},\C}}}
\newcommand{\GmR}{{\G_{{\rm m},\R}}}
\newcommand{\Gmk}{{\G_{{\rm m},\hs k}}}

\def\vk{\varkappa}
\def\eq{{\,\Leftrightarrow\,}}
\def\RD{{\rm RD}}

\def\pia{\pi^\alg}
\def\ev{{\rm ev}}

\begin{document}

\begin{abstract}
Let $X$ be a homogeneous space of a connected linear algebraic group $G$
defined over the field of complex numbers $\C$.
Let $x\in X(\C)$ be a point. We denote by $H$ the stabilizer of $x$ in $G$.
When $H$ is connected, we compute the topological fundamental group $\pi_1^\top(X(\C),x)$.
Moreover, we compute the second homotopy group $\pi_2^\top(X(\C),x)$.
\end{abstract}

\maketitle

\section{Introduction}

The topological fundamental group of a {\em connected  linear algebraic group} over  $\C$ was defined algebraically
by Merkurjev \cite[Section 10.1]{Merkurjev} and by the author   \cite[Definition 1.3]{Bor-Memoir}.
A third definition was proposed by  Colliot-Th\'el\`ene \cite[Proposition-Definition 6.1]{CT}.
The definition of Colliot-Th\'el\`ene was generalized to the reductive group schemes
by Gonz\'alez-Avil\'es \cite[Definition 3.7]{GA}
and by the author and Gonz\'alez-Avil\'es  \cite[Definition 2.11]{BG}.
We wish to compute algebraically the topological fundamental group
of a {\em homogeneous space} $G/H$
of a connected linear algebraic group $G$ over $\C$.
In general, by the result of the author and Cornulier \cite{BC},
the topological fundamental group of the homogeneous space $G/H$ cannot be computed algebraically.
In this article we compute algebraically the topological fundamental group of $G/H$
under a certain connectedness condition  on $H$, in particular, in the case when $H$ is connected.
Moreover, we compute algebraically the second homotopy group of $G/H$
(without connectedness assumptions on $H$).

In this article, by a variety over an algebraically closed field $k$ of characteristic 0
we mean a separated integral scheme of finite type over $k$.

Let $X$ be a variety defined over $\C$.
Let $x\in X(\C)$.
Consider the pointed  topological  space $(X(\C),x)$ and its (topological) homotopy groups
$\pi_1^\top(X(\C),x)$ and $ \pi_2^\top(X(\C),x)$.
We write  $\pi_1^\top(X,x)$ for $\pi_1^\top(X(\C),x)$, and we write $\pi_2^\top(X,x)$ for $\pi_2^\top(X(\C),x)$.
Set
$$\Z(1)=\pi_1^\top(\GmC(\C),1)=\pi_1^\top(\C^\times,1)$$
where  $\GmC$ denotes the multiplicative group over $\C$.
Choose an element  $\ii\in\C$ such that $\ii^2=-1$.
There is a generator $\xi_\ii$ (depending on the choice of $\ii$) of $\Z(1)=\pi_1^\top(\C^\times,1)$
given  by the loop
$$
t\mapsto \exp\,2\pi \ii\, t\colon\quad [0,1]\to \C^\times.
$$
We obtain an isomorphism  (depending on the choice of $\ii$)
$$\Z\,\isoto\, \Z(1)\colon m\mapsto m\xi_\ii\quad\ \text{for}\ \,m\in\Z.$$
We assume that the group $\pi_1^\top(X,x)$ is abelian. Consider the abelian groups
\begin{equation}\label{e:-1}
\pi_n^\top(X,x)(-1)\coloneqq\Hom\big(\Z(1), \pi_n^\top(X,x)\big)\quad\ \text{for}\ \, n=1,2.
\end{equation}
There are isomorphisms (depending on the choice of $\ii$)
\begin{equation*}
\pi_n^\top(X,x)(-1)=\Hom\big(\Z(1),\pi_n^\top(X,x)\big)\, \longisoto\,
   \pi_n^\top(X,x),\quad \ \phi\mapsto\phi(\xi_\ii).
\end{equation*}

\noindent
{\em Notation.}\ \
Let $k$ be an algebraically closed field of characteristic 0.
Let $H$ be a linear algebraic group defined over $k$.
We use the following notation:
\begin{itemize}
\item $H^0$ is the identity component of $H$;

\item $\pi_0(H)=H/H^0$, which is a finite $k$-group;

\item $H^\uu$ is the unipotent radical of $H^0$;

\item $H^\red=H^0/H^\uu$, which is a connected reductive group;

\item $H^\sss=[H^\red, H^\red]$  (the commutator subgroup of $H^\red$), which is semisimple;

\item $H^\ssc$ is the universal cover of $H^\sss$; it is simply connected;

\item $H^\tor=H^\red/H^\sss$, which is a torus;

\item $H^\ssu=\ker\big[H^0\to H^\tor\big]$; it fits into a short exact sequence
\begin{equation}\label{e:ssu}
1\to H^\uu\to H^\ssu\to H^\sss\to 1.
\end{equation}

\end{itemize}
Observe that $H^\tor$ is the largest toric quotient of  $H^0$ and that  $H^\ssu$ is connected and has no nontrivial characters.

If $T$ is a torus over $k$, we write   $T_*$
for the cocharacter group of $T$, that is,
 $$T_*=\Hom_k(\Gmk, T)$$
where $\Gmk$ denotes the multiplicative group over $k$.
Then $T_*\cong\Hom\big(\That,\Z\big)$ where  $\That=\Hom_k(T, \Gmk)$ is the character group of $T$.

Let $T_H^\red\subseteq H^\red$ be a maximal torus.
Consider the composite homomorphism
\[\rho\colon H^\ssc\onto H^\sss\into H^\red\]
and set $T_H^\ssc=\rho^{-1}(T_H^\red)$,  which is a maximal torus of $H^\ssc$.
Following \cite[Section 1]{Bor-Memoir},
we define the {\em algebraic fundamental group}
\[\pia_1(H)\coloneqq \coker \big[(T_H^\ssc)_*\to (T_H^\red)_*\big].\]
This group is well defined (does not depend on the choice of $T_H^\red$
up to a transitive system of isomorphisms); see \cite[Lemma 1.2]{Bor-Memoir}.

Let $X$ be a homogeneous space of a connected linear algebraic group  $G$
defined over an algebraically closed field $k$ of characteristic 0.
Choose a $k$-point $x\in X(k)$ and set  $H=\mathrm{Stab}_G(x)$.
Denote by $H^\mult$ the largest quotient group of  $H$ of multiplicative type.
Set $H^\kerchar:=\ker \big[H\to H^\mult\big]$;
then $H^\kerchar$ is the intersection
of the kernels  of all  characters $\chi\colon H\to\Gmk$ of $H$.
We consider the following conditions on $G$ and $H$:
\begin{enumerate}[\upshape (a)]
    \item  $\Pic(G)=0$,
    \item $H^\kerchar$ is connected.
\end{enumerate}
Observe that  (a) is satisfied if and only if  $G^\sss$ is simply connected
(see  Corollary \ref{c:Pic-G-Gss} below),
and that  (b) is satisfied if $H$ is connected or abelian.

Denote  $\Ghat:=\Hom(G,\Gmk)$ and $\Hhat:=\Hom(H,\Gmk)$; these groups are clearly abelian.
We write
$$
\Ext^0\big(\Ghat\to\Hhat,\Z\big):=\Ext^0\big(\big[\Ghat\labelt{i^*}\Hhat\big\rangle,\hs\Z\hs\big)
$$
where $\big[\Ghat\labelt{i^*}\Hhat\big\rangle$ is a complex
with $\Ghat$ in degree $0$ and $\Hhat$ in degree $1$.
The homomorphism $i^*$ is induced by the inclusion $i\colon H\into G$.
See Section \ref{s:extensions} for the definition of  $\Ext^0$.

\begin{theorem}\label{thm:main}
Let  $X$ be a homogeneous space of a connected linear algebraic group $G$ over $\C$.
Let $x\in X(\C)$,
and set $H=\mathrm{Stab}_G(x)$.
Assume that $\Pic(G)=0$ and that $H^\kerchar$ is connected.
Then the group $\pit_1(X,x)$ is abelian and
there exists a canonical isomorphism of abelian groups
$$\Ext^0\big(\Ghat\to\Hhat,\Z\big)\,\longisoto\,\pi_1^\top(X,x)(-1).$$
\end{theorem}

\begin{corollary}\label{cor:main}
Under the hypotheses of Theorem \ref{thm:main}:
\begin{enumerate}[\upshape (i)]
\item There is a canonical exact sequence
\begin{equation*}
\Hom\big(\Hhat,\Z\big)\labelto{i_*}\Hom\big(\Ghat,\Z\big)\lra\pi_1^\top(X,x)(-1)\lra\Hom\big(\Hhat_\tors,\Q/\Z\big)
\end{equation*}
where $i\colon H\hookrightarrow G$ is the inclusion homomorphism,
and $\Hhat_\tors$ denotes the torsion subgroup of $\Hhat$.

\item If, moreover, the subgroup $H$ is connected,
then the  exact sequence  of (i)  induces a canonical isomorphism
$$
\coker \big[H^\tor_*\labelto{i_*}G^\tor_*\big] \,\longisoto\, \pi^\top_1(X,x)(-1).
$$
\end{enumerate}
\end{corollary}

Observe that  Corollary \ref{cor:main}(ii) is
a more explicit version of \cite[Theorem 8.5(i)]{BvH2}.

\begin{theorem}\label{thm:pi2}
Let $X$ be a homogeneous space of a connected linear algebraic group $G$ over $\C$.
We do not assume that  $\Pic(G)=0$.
Let $x\in X(\C)$, and set $H=\Stab_G(x)$.
Let $i\colon H\into G$ denote the inclusion homomorphism.
\begin{enumerate}[\upshape (i)]

\item If $H$ is connected, then the group $\pit_1(X,x)$ is abelian and there is a canonical isomorphism
\[\coker\big[\pia_1(H)\labelto{i_*}\pia_1(G)\big] \longisoto \pit_1(X,x)(-1).\]

\item Even without the assumption that $H$ is connected, we have a canonical isomorphism
\[\pit_2(X,x)(-1)\longisoto \ker\big[\pia_1(H)\labelto{i_*}\pia_1(G)\big].\]
\end{enumerate}
\end{theorem}

Observe that Theorem \ref{thm:pi2}(i) is stronger than Corollary \ref{cor:main}(ii)
because in  Theorem \ref{thm:pi2}(i) we do not assume that $\Pic(G)=0$.
Moreover,  our Theorem \ref{thm:pi2}(ii) is stronger than Theorem 8.5(ii) of \cite{BvH2},
where only $\pi_2^\top(X,x)(-1)$ modulo torsion was computed.

\begin{remark}
Assume that our pair $(G,X)$ comes from some pair $(G_0,X_0)$
defined over the field of real numbers $\R$:
$(G,X)=(G_0,X_0)\times_\R\C$.
Here $G_0$ is a connected linear algebraic group over $\R$,
and  $X_0$ is a real algebraic variety on which $G_0$ acts over $\R$.
Let $\Gamma=\Gal(\C/\R)=\{1,\gamma\}$ denote the Galois group of $\C$ over $\R$
where $\gamma$ denotes complex conjugation.
Assume that $X_0$ has an $\R$-point $x_0$,
that is, a point $x_0\in X_0(\C)=X(\C)$ such that $^\gamma\hm x_0=x_0$.
Since $\Gamma$ acts on $X_0(\C)=X(\C)$ continuously and preserves $x_0$,
it naturally acts on the abelian groups $\pi_1^\top(X,x_0)$ and $\pi_2^\top(X,x_0)$.
Write $\Z(1)=\pit_1(\GmR(\C),1)$ and define abelian groups
$\pit_n(X,x)(-1)$ for $n=1,2$   by formula \eqref{e:-1}.
Then $\Gamma$ acts (nontrivially) on $\Z(1)$, and so
it acts on $\pit_n(X,x)(-1)$.
Note that $\pit_n(X,x_0)(-1)$ and $\pit_n(X,x_0)$ are isomorphic as abelian groups,
but in general they are not isomorphic as $\Gamma$-modules.

Let $H_0=\Stab_{G_0}(x_0)$; then $H_0$ is a linear algebraic $\R$-group.
Set $H=H_0\times_\R\C$.
The Galois group $\Gamma$ naturally  acts on $\Hhat$ and on $\Ghat$.
We can choose the torus $T_H^\red$ to be defined over $\R$;
then $\Gamma$ acts on $\pia_1(H)$, and similarly it acts on $\pia_1(G)$.
Now the canonical isomorphisms of abelian groups in Theorem \ref{thm:main},
Corollary \ref{cor:main}, and Theorem \ref{thm:pi2}
are actually isomorphisms of $\Gamma$-modules.
\end{remark}

Related results were obtained by Demarche \cite{Demarche}
over an algebraically closed field $k$ {\em of arbitrary characteristic.}

In the case $k=\C$, a result of \cite{Demarche} gives the profinite completion
of the group $\pi_1^\top(X,x)(-1)$.
To be more precise, let $\Zhat$ denote the profinite completion of $\Z$.
Let $k$ be an algebraically closed field {\em of characteristic} 0.
Let $X$ be a $k$-variety and $x\in X(k)$ be a $k$-point.
We write $\piet(X,x)$ for the \'etale fundamental group of $(X,x)$.
Set
\[\Zhat(1)=\piet(\Gmk,1),\qquad \piet(X,x)(-1)=\Hom\big(\Zhat(1),\piet(X,x)\big).\]
Then Theorem 1.4 of \cite{Demarche} in the case of characteristic 0
says that under the hypotheses of our Theorem \ref{thm:main},
there is a canonical isomorphism
\[\piet(X,x)(-1)\,\longisoto\, \Ext^0\big(\Ghat\to\Hhat,\Z\big)\otimes_\Z \Zhat.\]
When $k=\C$, the profinite  group $\piet(X,x)(-1)$ is isomorphic to the profinite completion
$\pi_1^\top(X,x)(-1)^\wedge$ of the finitely generated abelian group $\pi_1^\top(X,x)(-1)$,
and we obtain an isomorphism
\begin{equation}\label{e:Demarche}
\pi_1^\top(X,x)(-1)^\wedge\ \simeq\ \Ext^0\big(\Ghat\to\Hhat,\Z\big)\otimes_\Z \Zhat.
\end{equation}
Since $\pi_1^\top(X,x)(-1)$ and $\Ext^0\big(\Ghat\to\Hhat,\Z\big)$ are finitely generated abelian groups,
it follows from \eqref{e:Demarche} that they are isomorphic.
However, from \eqref{e:Demarche} we do not immediately obtain the  canonical isomorphism of  our Theorem \ref{thm:main}.
Thus in our case $k=\C$, we obtain a stronger result by elementary methods
(we use only the exact sequence of homotopy groups associated to a fibration).

Observe that the constructions of the corresponding isomorphisms
in Theorem \ref{thm:main} of the present article and in \cite{Demarche} are based on the same principles
(these constructions first appeared in the preprint \cite{BD} by Demarche and the author).
The difference is that we use technical tools of the classical (topological) homotopy theory,
whereas Demarche \cite{Demarche} uses tools of the \'etale homotopy theory.

The plan of the article is as follows.
In Section \ref{s:extensions} we discuss the functor $\Ext^0$.
In Section \ref{s:Picard}, for a connected $\C$-group $G$ we compare the conditions $\Pic(G)=0$ and $\pit_1(G^\ssu)=1$.
In  Section \ref{s:Aux} we recall constructions of auxiliary groups and  homogeneous spaces.
In Section \ref{s:Proofs} we prove Theorem \ref{thm:main} and Corollary \ref{cor:main}.
In Section \ref{s:pi2} we prove Theorem \ref{thm:pi2}.

\section{The functor $\Ext^0$}

\label{s:extensions}
In this section we consider the functor $\Ext^0(A^0\labelto\alpha A^1,\,\Z)$
where $\alpha\colon A^0\to A^1$ is a homomorphism of abelian groups.

Let $\Kbul$ be a bounded complex in an abelian category $\AA$, and let  $B$ be an object of  $\AA$.
Define
$$
\Ext_\AA^i(\Kbul,B):=\Hom_{D^b(\AA)}(\Kbul,B[i])
$$
where $D^b(\AA)$ is the derived category of bounded complexes in $\AA$,
and $B[i]$ is the complex consisting  of one object $B$ in degree $-i$.
If $A$ is an object of $\AA$, we have
$$
\Ext_\AA^i(A[0],B)=\Hom_{D^b(\AA)}(A[0],B[i])=:\Ext_\AA^i(A,B);
$$
see Gelfand and Manin \cite[Definition III.5.3]{GM}.
By definition, $\Ext_\AA^0(A,B)=\Hom_\AA(A,B)$.

We consider the category $\mathcal Ab$ of abelian groups and write $\Ext^i$ for  $\Ext^i_{\mathcal Ab}$.
Let $A$ be a finitely generated abelian group.
We write $A_\tors$ for the torsion subgroup of $A$,
and we set $A_\tf:=A/A_\tors$; then $A_\tf$ is a torsion-free finitely generated abelian group,
and therefore, it is a free abelian group.
It is clear that
 $$\Ext^0(A,\Z)=\Hom(A,\Z)=\Hom(A_\tf,\Z).$$

\begin{lemma}[well-known]\label{lem:Ext1}
For a finitely generated abelian group $A$, we have
$$\Ext^1(A,\Z)\cong\Hom(A_\tors,\Q/\Z).$$
\end{lemma}

\begin{proof}
From the injective resolution of $\Z$
$$
0\to\Z\to\Q\to\Q/\Z\to 0
$$
we obtain an exact sequence
\[  \Hom(A,\Q)\,\to\, \Hom(A,\Q/\Z)\,\to\, \Ext^1(A,\Z)\,\to\, \Ext^1(A,\Q)=0,\]
whence
$$
\Ext^1(A,\Z)\cong\coker\big[\Hom(A,\Q)\to\Hom(A,\Q/\Z)\big].
$$
We have
\begin{multline*}
\coker\big[\Hom(A,\Q)\to\Hom(A,\Q/\Z)\big]=\coker\big[\Hom(A_\tf,\Q)\to \Hom(A,\Q/\Z)\big]\\
                      =\coker\big[\Hom(A_\tf,\Q/\Z)\to \Hom(A,\Q/\Z)\big]\cong \Hom(A_\tors,\Q/\Z),
\end{multline*}
which proves the lemma.
\end{proof}

\begin{corollary}\label{c:Ext0}
For a homomorphism of finitely generated abelian groups  $\alpha\colon A^0\to A^1$,
the abelian group $\Ext^0(A^0\to A^1,\Z)$ fits into a canonical exact sequence
\begin{multline*}
\Hom(A^1,\Z)\labelt{\alpha^*}\Hom(A^0,\Z)\to
\Ext^0(A^0\to A^1,\Z)\to\\
\Hom(A^1_\tors\hs,\hs\Q/\Z)\labelto{\alpha^*}\Hom(A^0_\tors\hs,\hs\Q/\Z).
\end{multline*}
\end{corollary}

\begin{proof}
The short exact sequence of complexes
\[0\to (0\to A^1)\to (A^0\to A^1)\to (A^0\to 0)\to 0\]
gives rise to an exact sequence
\[\Hom(A^1,\Z)\labelt{\;\alpha^*}\Hom(A^0,\Z)\to
\Ext^0(A^0\to A^1,\Z)\to
\Ext^1(A^1,\Z)\to\Ext^1(A^0,\Z).\]
Applying  Lemma \ref{lem:Ext1}, we obtain the exact sequence of the corollary.
\end{proof}

Let $\alpha\colon A^0\to A^1$
be a homomorphism of finitely generated  abelian groups where $A^0$ is torsion-free.
We wish to compute  $\Ext^0(A^0\to A^1,\Z)$.

Choose a surjective homomorphism  $\varphi^1\colon B^1\to A^1$
where $B^1$ is a finitely generated {\em torsion-free} abelian group.
Set
\[B^0=A^0\times_{A^1} B^1=\big\{(a_0,b_1)\in A^0\times B^1\ \,\big|\ \,\alpha(a_0)=\varphi^1(b_1)\big\}.\]
We have a canonical homomorphism
\[\beta\colon B^0\to B^1,\quad (a_0,b_1)\mapsto b_1.\]

\begin{lemma}
$\Ext^0(A^0\to A^1,\Z)\,\cong\,
   \coker\big[\Hom(B^1,\Z)\labelto{\beta^*}\Hom(B^0,\Z)\big].$
\end{lemma}

\begin{proof}
We have a commutative diagram
\[
\xymatrix{
B^0\ar[r]^-\beta\ar[d]_-{\varphi^0}   &B^1\ar[d]^-{\varphi^1} \\
A^0\ar[r]^-\alpha                     &A^1
}
\]
where $\varphi^0(a_0,b_1)=a_0$.
We obtain a morphism of complexes
\[\varphi=(\varphi^0,\varphi^1)\colon\, (B^0\to B^1)\,\lra\,(A^0\to A^1).\]
One can easily check that $\varphi$ is a quasi-isomorphism, and therefore the induced homomorphism
\[\varphi^*\colon\Ext^0(A^0\to A^1,\Z)\,\lra\, \Ext^0(B^0\to B^1,\Z)\]
is an isomorphism.
Since both $A^0$ and $B^1$ are torsion-free, we see that $B^0$ is torsion-free as well.
By Corollary \ref{c:Ext0} we have an exact sequence
\[\Hom(B^1,\Z)\labelto{\;\alpha^*}\Hom(B^0,\Z)\lra \Ext^0(B^0\to B^1,\Z)\to 0,\]
and the lemma follows.
\end{proof}

\section{The Picard group and the topological fundamental group}
\label{s:Picard}

In this section we show that for a connected linear algebraic group $G$ over $\C$,
the condition $\Pic(G)=0$ is equivalent to $\pit_1(G^\ssu)=1$.

\begin{lemma}
\label{l:Pic-Gss}
Let $G$ be a connected linear algebraic group over an algebraically closed
field $k$ of characteristic 0. Then there are canonical isomorphisms
\[ \Pic(G)\longisoto \Pic(G^\ssu)\longisoto \Pic(G^\sss).\]
\end{lemma}

\begin{proof}
A short exact sequence of connected  linear algebraic groups over an algebraically closed
field of characteristic 0
\[1\to G'\to G\to G''\to 1\]
gives rise to an exact sequence
\begin{equation}\label{e:Pic-G''}
 \widehat{G'}\to\Pic(G'')\to \Pic(G)\to \Pic(G')\to 0;
\end{equation}
see Sansuc \cite[(6.11.2)]{Sansuc}.
Applying \eqref{e:Pic-G''} to the short exact sequence
\[1\to G^\ssu\to G\to G^\tor\to 1,\]
we obtain an exact sequence
\[\Pic(G^\tor)\to\Pic(G)\to \Pic(G^\ssu)\to 0.\]
Taking into account that $\Pic(G^\tor)=0$
(see, for instance, \cite[Lemma 6.9(ii)]{Sansuc}\hs),
we obtain an isomorphism $\Pic(G)\longisoto\Pic(G^\ssu)$.
Applying \eqref{e:Pic-G''} to the short exact sequence \eqref{e:ssu},
we obtain an exact sequence
\[ \widehat{G^\uu}\to\Pic(G^\sss)\to \Pic(G^\ssu)\to \Pic(G^\uu).\]
Taking into account that $\Pic(G^\uu)=0$
(because in characteristic 0 the variety of the unipotent group $G^\uu$
is isomorphic to the affine space ${\rm Lie}(G^\uu)$ via the logarithm map)
and that $\widehat{G^\uu}=0$,
we obtain an isomorphism $\Pic(G^\sss)\longisoto\Pic (G^\ssu)$,
which completes the proof.
\end{proof}

\begin{lemma}
\label{l:Pic-sc}
For $G$ as in Lemma \ref{l:Pic-Gss}, we have $\Pic(G)=0$
if and only if $G^\sss$ is simply connected.
\end{lemma}

\begin{proof}
By Lemma \ref{l:Pic-Gss},  we have $\Pic(G)\cong\Pic(G^\sss)$.
Set $\vk=\ker\big[G^\ssc\to G^\sss\big]$, which is a finite abelian group.
By \cite[Lemma 6.9(iii)]{Sansuc} we have  $\Pic(G^\sss)\cong\Hom(\vk_G, \Gmk)$,
whence we obtain that $\Pic(G)\cong\Hom(\vk_G,  \Gmk)$.
It follows that $\Pic(G)=0$ if and only if $\vk=\{1\}$, that is,
$G^\sss$ is simply connected, as required.

For another proof see Popov \cite{Popov}, who proves that for a connected linear algebraic group $G$
over an algebraically closed field $k$ {\em of arbitrary characteristic},
there is a canonical isomorphism $\Pic(G)\cong\Hom(\vk_G, \Gmk)$.
\end{proof}

\begin{lemma}
\label{l:sc-Lie-alg}
Let $G$ be a connected semisimple $\C$-group.
Then the Lie group $G(\C)$ is simply connected if and only if the algebraic group $G$ is simply connected.
\end{lemma}

\begin{proof}
Let $G$ be a connected semisimple $\C$-group.
Recall that $G$ is called {\em simply connected} if any isogeny $G'\to G$,
where $G'$ is a semisimple group, is an isomorphism.
Let $T\subset G$ be a maximal torus, and consider
the root datum $\RD(G,T)=(X,X^\vee, R, R^\vee)$;
see Springer  \cite[Sections 1 and 2]{Springer}.
A root datum is called {\em simply connected} if $X=P$
where $X=\That$ is the character group of $T$, and $P$ is the weight lattice.
Then the algebraic group $G$ is simply connected if and only if $\RD(G,T)$ is simply connected;
see \cite[Section 2.15]{Springer}.
On the other hand, also the Lie group $G(\C)$ is simply connected if and only if
the root datum  $\RD(G,T)$ is simply connected; see \cite[Chapter 3, Section 2.4, Theorem 2.6]{GOV}
or Conrad \cite[Proposition D.4.1]{Conrad}.
The lemma follows.
\end{proof}

\begin{corollary}\label{c:Pic-G-Gss}
For a connected linear algebraic group $G$ over $\C$,
the following assertions are equivalent:
\begin{enumerate}[\upshape (i)]
\item $\Pic(G)=0$;
\item the algebraic group $G^\sss$ is simply connected;
\item $\pi_1^\top(G^\sss)=1$;
\item $\pi_1^\top(G^\ssu)=1$.
\end{enumerate}
\end{corollary}

\begin{proof}
By Lemma \ref{l:Pic-sc}, we have (i)$\eq$(ii).

By Lemma \ref{l:sc-Lie-alg}, the Lie group $G^\sss(\C)$ is simply connected
if and only if  the algebraic group $G^\sss$ is simply connected. Thus (ii)$\eq$(iii).

From the short exact sequence \eqref{e:ssu} we obtain an isomorphism
$$\pi_1^\top(G^\ssu)\isoto\pi_1^\top(G^\sss).$$
Thus (iii)$\eq$(iv), which completes the proof of the lemma.
\end{proof}

\section{Auxiliary pairs}\label{s:Aux}

In this section we recall the constructions of auxiliary groups and homogeneous spaces
that we shall need for our proof of Theorem  \ref{thm:main}.
For a homogeneous space  $X$ of an algebraic $k$-group  $G$ satisfying the hypotheses
of Theorem \ref{thm:main}, we construct homogeneous spaces $Y$, $Z$, and $W$ of certain $k$-groups
($G_Y$, $G_Z$, and  $G_W$, respectively), with morphisms of pairs
$$(G,X) \leftarrow (G_Y, Y) \rightarrow (G_Z, Z) \rightarrow (G_W, W) $$
that will permit us  to prove  Theorem  \ref{thm:main} successively
for $W$, $Z$, $Y$, and finally for $X$.
These constructions already appeared in the preprint \cite{BD} and were published in \cite{Demarche};
for the reader's convenience we repeat them here.

{\em Construction of the homogeneous space  $Y$}.\ \
Let $X$ be a homogeneous space of a connected linear algebraic  $k$-group
defined over an algebraically closed field $k$ of characteristic 0.
Assume that  $\Pic(G)=0$.

Choose a $ k$-point $x\in X(k)$. We denote by $H$ the stabilizer of $x$ in $G$.
We do not assume that  $H$ is connected.

Let $H^\mult$ denote the largest quotient group of $H$ that is a group of multiplicative type.
Set $H^\kerchar=\ker\big[H\to H^\mult\big]$.
We have a canonical  homomorphism $H^\mult\to G^\tor$, which in general is not injective.

We choose an embedding $j\colon H^\mult\into Q$
of $H^\mult$ into a $k$-torus $Q$.
Consider the embedding
$$
j_*\colon H \to G\times_k Q, \quad\ h \mapsto (h,j(m_H(h)))\ \,\text{for}\ \,h\in H
$$
where $m_H\colon H \to H^{\mult}$ is the canonical  epimorphism.
Set
$$
G_Y=G\times_k Q,\quad H_Y=j_*(H)\subset G_Y,\quad Y=G_Y/H_Y,\quad y=1\cdot H_Y\in Y(k).
$$

The projection $\pi\colon G_Y = G \times Q \to G$ satisfies $\pi(H_Y) = H$, and
it induces a $G_Y$-equivariant map $\pi_*\colon Y\to X$ such that $\pi_*(y)=x$.
One can easily see that  $Y$ is a torsor over  $X$ under the torus $Q$.
We obtain a morphism of pairs
$$
 (G_Y,Y)\to (G,X).
$$
Observe that  the homomorphism  $H_Y^{\mult}\to G_Y^{\tor}$ is injective, and therefore
$$
H_Y\cap G_Y^\ssu=(H_Y)^\kerchar\cong H^\kerchar.
$$

{\em Construction of the homogeneous space $Z$.}
Set
$$G_Z=G_Y^\tor=G_Y/G_Y^\ssu\quad \text{where}\quad  G_Y^\ssu:=\ker\big[G_Y\to G_Y^\tor\big].$$
We have a canonical homomorphism $\mu\colon G_Y\to G_Z$.
Then $G_Z$ is a  $k$-torus and we have $\widehat{G_Z}=\widehat{G_Y}$.

The inclusion $i\colon H\into G$ induces a homomorphism $i^\mult\colon H^\mult\to G^\mult = G^\tor$.
We obtain an embedding
$$
\iota\colon H^\mult\into G^\tor\times_k Q,\quad\ h\mapsto (i^\mult(h),j(h))\ \,\text{for}\ \, h\in H^\mult.
$$
Set
$$
Z=Y/G_Y^\ssu=(G^\tor\times_k Q)/\iota(H^\mult);
$$
then we have a $G_Y$-equivariant map $\mu_*\colon Y\to Z$ whose fiber over the  $k$-point
$z:=\mu_*(y)\in Z(k)$ is isomorphic to
$$
G_Y^\ssu/(H_Y\cap G_Y^\ssu)\cong G^\ssu / H^\kerchar.
$$
The variety $Z$ is a homogeneous space of $G_Z$
with stabilizer $H_Z=H_Y^\mult\subset G_Y^\tor=G_Z$.
Note that
$$
\widehat{H_Z}=\widehat{H_Y^\mult}=\widehat{H_Y}.
$$
We have a natural morphism of pairs
$$
(G_Y,Y)\to (G_Z,Z).
$$

{\em Construction of the homogeneous space  $W$.}
We set  $G_W=G_Z/H_Z$, $W=Z$, $w=z$;
then $W$ is a principal homogeneous space of the torus $G_W$.
There is a natural morphism of pairs
$$
(G_Z,Z)\to (G_W, W).
$$

\section{Proof of Theorem \ref{thm:main}}
\label{s:Proofs}

In this section we prove Theorem \ref{thm:main} and
Corollary \ref{cor:main}.
We use lemmas of Section \ref{s:Picard} and the constructions of Section \ref{s:Aux}.

{\em Step} 1.\ \
First we treat the case of a principal homogeneous space $W$ of a  $k$-torus $G_W$.
Let $w\in W(\C)$ be a $\C$-point.
The map $G_W\to W$ defined by $g\mapsto g\cdot w$ is an isomorphism of $\C$-varieties,
and we have an induced isomorphism of groups
\[ \pi_1^\top(G_W,1)\isoto\pi_1^\top(W,w).\]
Since $G_W$ is a torus, the group $\pit_1(G_W)$ is abelian,
and therefore the group $\pit_1(W,w)$ is abelian as well.
We have an induced isomorphism of abelian groups
$$
\pi_1^\top(G_W,1)(-1)\isoto\pi_1^\top(W,w)(-1).
$$
Since
$$
\pi_1^\top(G_W,1)(-1)=(G_{W})_*\cong\Hom\big(\widehat{G_W},\Z\big)=\Ext^0\big(\widehat{G_W},\Z\big),
$$
we obtain a canonical isomorphism
$$
\pi_1^\top(W,w)(-1)\isoto\Ext^0\big(\widehat{G_W},\Z\big).
$$
This proves Theorem \ref{thm:main} for $(G_W,W)$.

{\em Step} 2.\ \
Assume that we have a  homomorphism of $\C$-tori $\gamma_\alpha\colon G_{W'}\to G_W$
and a  $\gamma_\alpha$-equivariant map of principal homogeneous spaces
$\alpha\colon W'\to W$ sending a $\C$-point $w'\in W'(\C)$ to a $\C$-point $w\in W(\C)$.
Then the following  diagram clearly  commutes:
\begin{equation}\label{e:diag-tori}
\begin{aligned}
\xymatrix@C=10mm{
\pi_1^\top(W',w')(-1)\ar[d]_-\cong \ar[r]^-{\alpha_*} &\pi_1^\top(W,w)(-1)\ar[d]^-\cong \\
\Ext^0\big(\widehat{G_{W'}},\Z\big)  \ar[r]^-{\ \,\gamma_{\alpha *}}   &\Ext^0\big(\widehat{G_{W}},\Z\big)
}
\end{aligned}
\end{equation}
where the vertical arrows are canonical isomorphisms of Step 1.

We have $Z=W$, whence $\pit_1(Z)$ is an abelian group.
Moreover, $G_Z/H_Z=G_W$, and the evident morphism of complexes
$$\widehat{G_W}\to \big[\widehat{G_Z}\to\widehat{H_Z}\big\rangle$$
is a quasi-isomorphism, whence
$$
\pi_1^\top(Z,z)(-1)=\pi_1^\top(W,w)(-1)\cong\Ext^0\big(\widehat{G_W},\Z\big)
\cong\Ext^0\big(\big[\widehat{G_Z}\to\widehat{H_Z}\big\rangle,\Z\big),
$$
and we obtain a canonical isomorphism
$$\pi_1^\top(Z,z)(-1)\,\longisoto\,\Ext^0\big(\big[\widehat{G_Z}\to\widehat{H_Z}\big\rangle,\Z\big).$$
This proves \ref{thm:main} for $(G_Z,Z)$.

{\em Step} 3.\ \
There is a  fibration $G^\ssu(\C)\to G^\ssu(\C)/H^\kerchar(\C)$ {\em with connected fiber}  $H^\kerchar(\C)$,
and hence a fibration exact sequence
$$
1=\pi_1^\top(G^\ssu)\lra \pi_1^\top(G^\ssu/H^\kerchar)\lra\pi_0(H^\kerchar)=1.
$$
In this sequence, we have $\pi_1^\top(G^\ssu)=1$ because $\Pic(G)=0$;
see Corollary  \ref{c:Pic-G-Gss}.
We see that $\pi_1^\top(G^\ssu/H^\kerchar)=1$.
But we have a fibration $\mu_*\colon Y(\C) \to Z(\C)$ with fiber
$G^\ssu(\C)/H^\kerchar(\C)$, and hence a fibration exact sequence
$$
1=\pi_1^\top(G^\ssu/H^\kerchar)\lra\pi_1^\top(Y,y)\labelto{\;\mu_*}\pi_1^\top(Z,z)
    \lra\pi_0(G^\ssu/H^\kerchar)=1.
$$
It follows that the homomorphism $\pi_1^\top(Y,y)\labelto{\;\mu_*}\pi_1^\top(Z,z)$ is an isomorphism,
whence the group $\pi_1^\top(Y,y)$ is abelian and we have an isomorphism
of abelian groups
\[\pi_1^\top(Y,y)(-1)\labelto{\mu_*}\pi_1^\top(Z,z)(-1).\]
Since $\widehat{G_Y}=\widehat{G_Z}$ and $\widehat{H_Y}=\widehat{H_Z}$\hs,
we deduce  Theorem \ref{thm:main} for $(G_Y,Y)$
from Theorem \ref{thm:main} for $(G_Z,Z)$.

\def\fone{{ \framebox{\makebox[\totalheight]{1}} }}
\def\ftwo{{ \framebox{\makebox[\totalheight]{2}} }}
\def\fthree{{ \framebox{\makebox[\totalheight]{3}} }}
\def\ffour{{ \framebox{\makebox[\totalheight]{4}} }}

{\em Step} 4.\ \
We have a torsor $\pi_*\colon Y\to X$  under the torus $Q$,
whence we obtain an exact sequence of groups
\[ \pi_1^\top(Q,1)\labelto{\lambda_*}\pi_1^\top(Y,y)\labelto{\pi_*}\pi_1^\top(X,x)\to 1\]
where the arrow $\lambda_*$ is induced by the map
\[\lambda\colon Q\to Y,\quad q\mapsto q\cdot y\ \,\text{for}\ \,q\in Q.\]
Since the group $\pit_1(Y,y)$ is abelian, so is the group $\pit_1(X,x)$,
and we obtain an exact sequence of abelian groups
\[\pi_1^\top(Q,1)(-1)\labelto{\lambda_*}\pi_1^\top(Y,y)(-1)\labelto{\pi_*}\pi_1^\top(X,x)(-1)\to 0.\]
We have a short exact sequence of complexes
$$
0\to \big(\Ghat\to\Hhat\big)\,\lra\, \big(\widehat{G_Y}\to\Hhat\big)\,\lra\,\big(\widehat{Q}\to 0\big)\to 0,
$$
whence we obtain an exact sequence
\begin{equation*}
\Ext^0\big(\widehat{Q},\Z\big)\,\lra\, \Ext^0\big(\widehat{G_Y}\to\Hhat,\Z\big)\,\lra\,
    \Ext^0\big(\Ghat\to\Hhat,\Z\big)\,\lra\, \Ext^1\big(\widehat{Q},\Z\big)= 0
\end{equation*}
(because by Lemma \ref{lem:Ext1} we have $\Ext^1\big(\widehat{Q},\Z\big)=\Hom\big(\widehat{Q}_\tors,\Q/\Z\big)=0$).
We obtain a diagram with exact rows
\begin{equation}\label{eq:Q}
\begin{aligned}
\xymatrix{
{\pi}_1^\top(Q,1)(-1) \ar[r]^-{\lambda_*} \ar[d]^-\cong\ar@{}[dr]|{\fone}
     &\pi_1^\top(Y,y)(-1)\ar[r]^-{\pi_*}\ar[d]^-\cong\ar@{}[dr]|{\ftwo}  &\pi_1^\top(X,x)(-1)\ar[r]\ar@{.>}[d]      &0 \\
\Ext^0\big(\widehat{Q},\Z\big)\ar[r]          &\Ext^0\big(\widehat{G_Y}\to\widehat{H_Y},\Z\big)\ar[r]
     &\Ext^0\big(\widehat{G}\to\widehat{H},\Z\big)\ar[r]  &0\, .
}
\end{aligned}
\end{equation}

We show that  rectangle $\fone$ commutes.
Consider the diagram
$$
\xymatrix@C=6mm{
{\pi}_1^\top(Q,1)(-1) \ar[r]^-{\lambda_*} \ar[d]^-\cong\ar@{}[dr]|{\fone}   &\pi_1^\top(Y,y)(-1)\ar[r]^-\cong\ar[d]^-\cong\ar@{}[dr]|{\fthree}
                                    &\pi_1^\top(Z,z)(-1)\ar[r]^-\cong\ar[d]^-\cong \ar@{}[dr]|{\ffour}      &\pi_1^\top(W,w)(-1)\ar[d]^-\cong   \\
\Ext^0\big(\widehat{Q},\Z\big)\ar[r]          &\Ext^0\big(\widehat{G_Y}\to\widehat{H_Y},\Z\big)\ar[r]^-\cong
&\Ext^0\big(\widehat{G_Z}\to\widehat{H_Z},\Z\big)\ar[r]^-\cong  &\Ext^0\big(\widehat{G_W},\Z\big) .
}
$$
By construction,  rectangles $\fthree$ and  $\ffour$ commute.
The commutative  diagram \eqref{e:diag-tori} shows
that the big rectangle $\fone\cup\fthree\cup\ffour$ commutes.
It follows that  rectangle $\fone$ commutes.

In the diagram with exact rows \eqref{eq:Q}, rectangle $\fone$ commutes,
which permits one to define the dotted arrow so that rectangle $\ftwo$ commutes.
Thus we obtain an isomorphism
\begin{equation}\label{eq:iso}
 \pi_1^\top(X,x)(-1)\,\longisoto\, \Ext^0\big(\widehat{G}\to\widehat{H},\Z\big),
\end{equation}
which a priori might depend on the embedding  $j\colon H^\mult\into Q$.

When constructing the torsor $Y\to X$ in Section \ref{s:Aux},
we constructed it from an embedding $j\colon H^\mult\into Q$.
If we choose another embedding $j'\colon H^\mult\into Q'$,
we obtain another torsor  $Y'\to X$ under $Q'$.
Set $Q''=Q\times_k Q'$ and denote by $j''\colon H^\mult\into Q''$ the diagonal embedding.
We obtain a torsor $Y''\to X$ under $Q''$ dominating  both $Y$ and $Y'$,
and from this fact we can easily see that the isomorphism \eqref{eq:iso}
does not depend on the choice  of the embedding $j\colon H^\mult\into Q$.
This completes the proof of Theorem \ref{thm:main}.
\qed

{\em Proof of Corollary \ref{cor:main}.}
The short exact sequence of complexes
$$
0\to \big[0\to \Hhat\big\rangle \,\lra\,\big[\Ghat\to\Hhat\big\rangle\,\lra\, \big[\Ghat\to 0\big\rangle\to 0
$$
induces a long exact sequence
\begin{equation}\label{eq:exacte-preuve}
\Ext^0\big(\Hhat,\Z\big)\,\to\, \Ext^0\big(\Ghat,\Z\big)\,\to\, \Ext^0\big(\Ghat\!\to\!\Hhat, \Z\big)\,
     \to\,\Ext^1\big(\Hhat,\Z\big)\,\to\,\Ext^1\big(\Ghat,\Z\big).
\end{equation}
We have  $\Ext^0\big(\Hhat,\Z\big)=\Hom\big(\Hhat, \Z\big)$ and  $\Ext^0\big(\Ghat,\Z\big)=\Hom\big(\Ghat,\Z\big)$.
By Lemma \ref{lem:Ext1}, we have $\Ext^1\big(\Ghat,\Z\big)\cong\Hom\big(\Ghat_\tors,\Q/\Z\big)=0$ and
$$
\Ext^1\big(\Hhat,\Z\big)\,\cong\,\Hom\big(\Hhat_\tors,\Q/\Z\big).
$$
By Theorem \ref{thm:main} we may write $\pi_1^\top(X,x)(-1)$
instead of  $\Ext^0\big(\Ghat\to\Hhat, \Z\big)$ in \eqref{eq:exacte-preuve}.
Thus we obtain from \eqref{eq:exacte-preuve} an exact sequence
\[\Hom\big(\Hhat,\Z\big)\labelt{i_*}\Hom\big(\Ghat,\Z\big)\,\to\,\pi_1^\top(X,x)(-1)\,
   \to\, \Hom\big(\Hhat_\tors,\Q/\Z\big)\,\to 0,\]
which proves assertion (i) of Corollary \ref{cor:main}.
If $H$ is connected, then $\Hhat=\widehat{H^\tor}$ is torsion-free,
and assertion (ii) follows from assertion (i).
\qed

\section{Proof of  Theorem \ref{thm:pi2}.}
\label{s:pi2}

Let $G$ and $H$ be as in Theorem \ref{thm:pi2}.
Let $T^\red_H\subseteq H^\red$ and $T^\ssc_H\subseteq H^\ssc$ be compatible maximal tori.
We recall the construction of an isomorphism
\begin{equation}\label{e:B98-H}
\ev_H\colon\, \pia_1(H)\coloneqq\pia_1(H^\red)\coloneqq\coker\big[(T_H^\ssc)_*\hs
    \to\hs (T_H^\red)_*\big]\,\longisoto\,\pi_1^\top(H)(-1)
\end{equation}
from \cite[Prop.~1.11]{Bor-Memoir}. Here we write $\pit_1(H)(-1)$ for $\pit_1(H^0)(-1)$.
A cocharacter $\nu\in(T_H^\red)_*$ gives a homomorphism
\[\GmC\to T_{H^\red}\to H^\red\]
and induces a homomorphism
\[\nu_*\colon\pit_1(\GmC)\to\pit_1(H^\red)\cong\pit_1(H).\]
If $\nu$ comes from $(T_H^\ssc)_*$\hs, then $\nu_*$ factors
via $\pit_1(H^\ssc)=1$ and hence is trivial.
Thus we obtain the desired homomorphism \eqref{e:B98-H}
\[
\pia_1(H^\red)
   \coloneqq\coker\big[(T_H^\ssc)_*\!\to\! (T_H^\red)_*\big]\hs\lra\hs
\Hom\big(\pit_1(\GmC),\pit_1(H)\big)\eqqcolon \pi_1^\top(H)(-1).
\]
By \cite[Prop.~1.11]{Bor-Memoir} this homomorphism is an isomorphism.
Similarly we construct an isomorphism
\begin{equation}\label{e:B98-G}
\ev_G\colon \pia_1(G^\red)\longisoto\pi_1^\top(G)(-1).
\end{equation}
Since the group $\pit_1(G)$ is isomorphic (non-canonically)
to the abelian group $\pia_1(G)$, it is abelian.
Note that the following diagram commutes:
\[
\xymatrix@C=10mm{
\pia_1(H) \ar[r]^-{i_*}\ar[d]_-{\ev_H}^-\cong  &\pia_1(G)\ar[d]^-{\ev_G}_-\cong \\
\pit_1(H)(-1) \ar[r]^-{\,i_*}             &\pit_1(G)(-1) .
}
\]

It is well known that for a connected linear algebraic group $G$ over $\C$, we have $\pit_2(G)=0$.
Indeed, we may assume that $G$ is simple and simply connected.
Let $K$ be a maximal compact subgroup of $G(\C)$.
We have $\pi_2^\top(K)=0$ by Cartan's theorem; see Borel \cite[Theorem 1(2)]{Borel}.
The Cartan decomposition for $G(\C)$ shows that $G(\C)$ is homotopically equivalent to $K$;
see \cite[Chapter 7, Section 3.2, Corollaries 1 and 5 of Theorem 3.2]{GOV}
or Knapp \cite[Theorem 6.31(c),(g)]{Knapp}.
It follows that  $\pi_2^\top(G)\cong\pi_2^\top(K)=0$.

We have a fibration
\[G\to X,\quad\ g\mapsto g\cdot x\ \ \text{for}\ \, g\in G\]
with fiber $H$, where we write $G$ for $G(\C)$, $H$ for $H(\C)$, and so on.
This fibration gives rise to a homotopy exact sequence
\begin{equation}\label{e:f-exact}
0=\pit_2(G)\to\pit_2(X,x)\to\pit_1(H)\to\pit_1(G)\to\pit_1(X,x)\to\pit_0(H);
\end{equation}
see, for instance, Hatcher \cite[Theorem 4.41]{Hatcher}.

First assume that $H$ is connected, that is, $\pit_0(H)=1$.
Since in the sequence \eqref{e:f-exact}
the group $\pit_1(G)$ is abelian, so is $\pit_1(X,x)$.
Applying to a part of sequence \eqref{e:f-exact}
the twisting functor $\Hom(\pit_1(\GmC),\,\cdot\,)$,
we obtain an exact sequence
\[\pit_1(H)(-1)\labelt{i_*}\pit_1(G)(-1)  \to\pit_1(X,x)(-1)\to 0,\]
which by \eqref{e:B98-H} and \eqref{e:B98-G} gives Theorem \ref{thm:pi2}(i).

Now we do not assume that $H$ is connected.
Applying to a part of sequence \eqref{e:f-exact}
the twisting functor $\Hom(\pit_1(\GmC),\,\cdot\,)$,
we obtain an exact sequence
\[0\to\pit_2(X,x)(-1)\to\pit_1(H)(-1)\labelt{i_*}\pit_1(G)(-1),\]
which by \eqref{e:B98-H} and \eqref{e:B98-G} gives Theorem \ref{thm:pi2}(ii).
\qed

\noindent
{\sc Acknowledgements:}\ \ The author thanks Cyril Demarche, Vladimir Hinich, and Tam\'as Szamuely
for very helpful discussions and  email correspondence.
We thank the anonymous referee for thorough reading the paper
and for his/her comments, which helped us to improve the exposition.

\end{document}